\documentclass[reqno]{amsart}

\usepackage{epic}
\usepackage{arydshln}
\usepackage{color}
\usepackage{tikz}

\usepackage{amsmath}
\usepackage{amssymb}
\usepackage{amsfonts}
\usepackage{amsthm}
\theoremstyle{definition}
\newtheorem{defn}{Definition}%[section]
\newtheorem{thm}{Theorem}%[section]
\newtheorem{cor}{Corollary}%[section]
\newtheorem{lem}{Lemma}%[section]
\newtheorem{rmk}{Remark}%[section]
\newtheorem{cnj}{Conjecture}%[section]

\newcommand{\ind}{\mathrm{ind}}
\newcommand{\cind}{\mathrm{cind}}
\newcommand{\op}{\mathrm{op}}
\newcommand{\pa}{\mathcal{P}}
\newcommand{\opa}{\mathcal{O}}

\newcommand{\dis}{\mathcal{D}}
\newcommand{\lam}{\lambda}

\newcommand{\be}{\bar{e}}
\newcommand{\bo}{\bar{o}}
\newcommand{\tc}{\tilde{c}}

\newcommand{\Mod}[1]{\ (\mathrm{mod}\ #1)}

\title{Index of seaweed algebras and integer partitions}
\author{Seunghyun Seo}
\address{Department of Mathematics Education, Kangwon National University,
Chuncheon, Kangwon-do 24341, Republic of Korea} \email{shyunseo@kangwon.ac.kr}

\author{Ae Ja Yee}
\address{Department of Mathematics, The Pennsylvania State University,
University Park, PA 16802, USA} \email{auy2@psu.edu}

\begin{document}

\maketitle

\footnotetext[1]{The first author was supported by Basic Science Research Program through the National Research Foundation of Korea(NRF) funded by the Ministry of Education(120180215).} \vspace{0.5in}
\footnotetext[2]{The second author was partially supported by a grant ($\#$633963) from the Simons Foundation.} \vspace{0.5in}
\footnotetext[3]{2010 AMS Classification Numbers: Primary, 05A17; Secondary, 11P81.}

\noindent{\footnotesize{\bf Abstract.}  The index of a Lie algebra is an important algebraic invariant.  In 2000, Vladimir Dergachev and Alexandre Kirillov  defined seaweed subalgebras of $\mathfrak{gl}_n$ (or $\mathfrak{sl}_n$) and provided a formula for the index of a seaweed algebra using a certain graph, so called a meander.
\\ \indent
In a recent paper, Vincent Coll, Andrew Mayers, and Nick Mayers defined a new statistic for partitions, namely  the index of a partition, which arises from seaweed Lie algebras of type A. At the end of their paper, they presented an interesting conjecture, which involves integer partitions into odd parts. Motivated by their work, in this paper, we exploit various index statistics and the index weight generating functions for partitions.  In particular, we examine their conjecture by considering the generating function for partitions into odd parts.  We will also reprove another result  from their paper using generating functions.}

\section{Introduction}

An integer partition $\lambda$ is a weakly decreasing finite sequence of positive integers $(\lambda_1,\lambda_2,\ldots,\lambda_r)$ \cite{gea1}. The $\lambda_i$'s are called the parts of $\lambda$ and the sum of the parts is called the weight of $\lambda$. For a positive integer $n$, if the weight of $\lambda$ is $n$, then $\lambda$ is called a partition of $n$ and denoted by $\lambda\vdash n$. It is the convention that the empty sequence is the only partition of $0$.

In the theory of partitions, parity has played an important role. In \cite{gea270}, G. E.  Andrews gave an extensive study on parity questions in partition identities. The very first example of Andrews mentioned in \cite{gea270} is the following theorem by Euler.
\begin{quote}
{\bf Euler's partition identity:} The number of partitions of any positive integer $n$ into distinct parts equals the number of partitions of $n$ into  {\it odd} parts.
\end{quote}

Recently, the parity of a certain statistic assigned to partitions into odd parts was considered in the study of seaweed Lie algebras. In \cite{collmayersmayers}, V. Coll, A. Mayers, and N. Mayers defined a new statistic for partitions, namely  the index of a partition, which arises from seaweed Lie algebras of type A. At the end of their paper, they presented a conjecture on the difference between the number of partitions of $n$ into odd parts with an odd index and the number of partitions of $n$ into odd parts with an even index. The definition of the index statistic requires some other concepts from Lie algebra, so we defer it to  Section~\ref{sec2}.  Let $E_{\ind}(n)$ denote the number of partitions of $n$ into odd parts with an odd index minus the number of partitions of $n$ into odd parts with an even index. 

\begin{cnj}[Coll, Mayers, Mayers \cite{collmayersmayers}] \label{conj1}
$$\sum_{n\ge 0} |E_{\ind}(n)| q^n=\prod_{n=1}^{\infty} \frac{1}{1+(-1)^n q^{2n-1}}. $$ 
\end{cnj}

The main purpose of this paper is to study this counting function $E_{\ind}(n)$.  The following theorem is one of our main results.
\begin{thm} \label{thm1}
We have
$$
\sum_{n\ge 0}(-1)^{ \lceil \frac{n}{2} \rceil }E_{\ind}(n) q^n =\prod_{n=1}^{\infty} \frac{1}{1+(-1)^n q^{2n-1}}.
$$
\end{thm}

It follows from Theorem~\ref{thm1} that Conjecture~\ref{conj1} is equivalent to $(-1)^{ \lceil \frac{n}{2} \rceil }E_{\ind}(n)\ge 0$, i.e., the infinite product on the right hand side has nonnegative coefficients.  The non-negativity seems very difficult to prove, and their conjecture is still open.

Along with the index statistic arising in Conjecture~\ref{conj1}, other index statistics can be defined on partitions. In particular, V. Coll, A. Mayers, and N. Mayers considered a case that is associated with partitions having parts of size one only, and they related that case to partitions into two-colored parts. In Section~\ref{sec5}, we will treat this index statistic and exploit the generating functions for various subsets of partitions.

This paper is organized as follows. In Section~\ref{sec2}, we will recollect some results on seaweed algebras from the literature. In Section~\ref{sec3}, we give the definition of the index of a partition and prove Theorem~\ref{thm1} by considering a weight generating function for partitions.  In Section~\ref{sec4}, we will find the generating function weighted with another index statistic, which will reprove a result of V. Coll, A. Mayers, and N. Mayers. In Section~\ref{sec5}, we will investigate further the index statistic considered in Section~4. We will  then provide some remarks and further conjectures in the last section.

\section{Preliminaries}\label{sec2}

A seaweed algebra is a subalgebra of $\mathfrak{sl}_n$, which was first introduced by V. Dergachev and A. Kirillov in \cite{dergachevkirillov}. The index of a Lie algebra is an important algebraic invariant in the study, which was given by J. Dixmier \cite{dixmier}, and  in the same paper, V. Dergachev and A. Kirillov  provided a combinatorial algorithm to compute the index of a seaweed algebra using a certain graph, namely the meander of a seaweed algebra.  Rather than giving the definitions of seaweed algebras and the index of Lie algebras, here we give a short description of the algorithm of V. Dergachev and A. Kirillov.  

A seaweed algebra can be defined by two compositions $\lambda$ and $\mu$ of $n$. In this case, we say that the seaweed is of type $\frac{\lambda}{\mu}=\frac{\lambda_1| \cdots | \lambda_{r}}{\mu_1|\cdots | \mu_s}$. 

The meander of a seaweed of type $\frac{\lambda_1| \cdots | \lambda_{r}}{\mu_1|\cdots | \mu_s}$ is a graph with $n$ vertices whose edges are given by $\lambda$ and $\mu$ as follows. 
\begin{itemize}
\item Place the $n$ vertices in a row.
\item First, partition the $n$ vertices into blocks of sizes $\lambda_1,\ldots, \lambda_r$, and let us call the vertices in the $i$-th block $v_{i,1},\ldots, v_{i,\lambda_i}$.
\item Connect $v_{i,j}$ and $v_{i,\lambda_{i}+1-j}$  by an edge for $1\le j\le \lfloor \lambda_i/2 \rfloor$. We draw these edges above the vertices and call them top edges.
\item Repeat this process with $\mu$. This time, we draw edges below the vertices and call them bottom edges.
\end{itemize}

Figure~\ref{fig1} shows the meander of a seaweed of type $\frac{3|2|1|1}{4|3}$.

\begin{figure}[ht]
\begin{tikzpicture}
\foreach \x in {0,...,6}
	\filldraw (\x*1, 0) circle (.5mm);
\draw (0,0) to[out=90, in=90] (2,0);
\draw (3,0) to[out=60, in=120] (4,0);
\draw (0,0) to[out=-60,in=-120] (3,0); 
\draw (1,0) to[out=-60,in=-120] (2,0); 
\draw (4,0) to[out=-60,in=-120] (6,0); 
\end{tikzpicture}
\caption{The meander of a seaweed of type $\frac{3|2|1|1}{4|3}$}\label{fig1}
\end{figure}
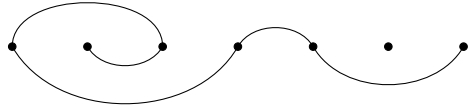 

In \cite{dergachevkirillov}, V. Dergachev and A. Kirillov showed that
\begin{equation}
\ind_{\mu}(\lam)=2C+P-1, \label{index}
\end{equation}
where $C$ and $P$ count the number of cycles and paths in the meander of a seaweed of type $\frac{\lambda}{\mu}$, respectively.  Here, we regards a vertex  of degree $0$, i.e., an isolated vertex, as a path.  
In Figure~\ref{fig1}, the index of the seaweed algebra  is $1$. 

Note that in the meander associated with a seaweed of type $\frac{\lambda}{\mu}$, a vertex is connected to another vertex by a top edge (resp. a bottom edge) unless it is the middle point in a block of size equal to an odd part of $\lambda$ (resp. $\mu$). Since in a simple graph $G$, if a path is of length at least one, then each of its end vertices has degree $1$, and an isolated vertex has degree $0$,  it can be easily shown that in the meander associated to a seaweed of type $\frac{\lambda}{\mu}$, 
\begin{equation}
P=\frac{\op(\lambda)+\op(\mu)}{2},  \label{p}
\end{equation}  
where $\op(\lam)$ count the number of odd parts in $\lam$. Thus, by \eqref{index} and \eqref{p}, we get
\begin{equation}
\ind_{\mu}(\lam)\equiv \frac{\op(\lam)+\op(\mu)}{2}-1 \pmod{2}. \label{eq1}
\end{equation}

\section{Parity difference of $\ind_n(\lam)$}\label{sec3}

We first define the index of a partition $\lambda$. 
\begin{defn} 
For a partition $\lambda$ of $n$, we define the index of $\lambda$ by the index of the seaweed of type $\frac{\lambda}{n}$ and denote it by $\ind(\lambda)$, i.e.,
\begin{equation*}
\ind(\lambda):=\ind_n(\lambda).
\end{equation*}
For $n=0$, if $\lambda$ is the empty partition, we define $\ind_n(\lambda)$ by $-1$. 
\end{defn}

In this section, we will consider the generating function for partitions with weight associated with index.  Let $\pa$ be the set of all partitions. For any subset $S$ of  $\pa$ and nonnegative integer $n$, define
$$
S_n := \{ \lam \in S \,|\,  \lam \vdash n \}.
$$

\begin{defn}
Given a subset $S$ of $\pa$, we define
\begin{align*}
    %e_n=
    e_{S}(n)&:= \# \{\lam\in S_n \, |\,   \mbox{$\ind(\lam)$ is even}\},\\
    %o_n=
    o_{S} (n)&:=\# \{\lam\in S_n \, |\,   \mbox{$\ind(\lam)$ is odd}\}.
\end{align*}
\end{defn}

We now consider the difference $o_{S}(n)-e_{S}(n)$. 
\begin{lem}  \label{lem1}
For any $n\ge 0$,
\begin{align*}
& o_S(n)-e_{S}(n) \\
&=\begin{cases}
\# \{\lam\in S_n \, |\,  \op(\lam) \equiv 0 \Mod{4}\}   -  \# \{\lam\in S_n \, |\,  \op(\lam) \equiv 2 \Mod{4}\} &\mbox{if $n$ is even,}\\ 
\# \{\lam\in S_n \, |\,  \op(\lam) \equiv 3 \Mod{4}\} - \# \{\lam\in S_n \, |\,  \op(\lam) \equiv 1 \Mod{4}\}  & \mbox{if $n$ is odd}.
\end{cases}
\end{align*}
\end{lem}

\begin{proof}
It easily follows from \eqref{eq1}.
\end{proof}

\begin{defn} \label{def:fF}
Given a subset $S$ of $\pa$, let 
$$ 
f_S (k,n):=\#\{\lam\in S_n \, |\,  \op(\lam)=k \}$$
and
$$
F_S (t,q):=\sum_{k,n\ge 0} f_S(k,n)t^k q^n.
$$
\end{defn}

From now on, we may omit the subscript $S$ in $o_S, e_S, f_S(k,n)$ and $F_S(t,q)$ if it is clear from the context. 
\begin{thm} \label{thm2}
We have
$$
F(i,-iq)
= \sum_{n\ge 0}(-1)^{ \lceil \frac{n}{2} \rceil }\big(o(n) -e(n)\big)q^n ,
$$
where $i=\sqrt{-1}$.
\end{thm}
\begin{proof}
Note that for a partition $\lambda$ of $n$, if  $n$ is even or odd, then $\op(\lambda)$ is even or odd, respectively. Thus, 
\begin{align}
F(t,q)
&=\sum_{\text{$k,n$:\,even}}f(k,n)t^k q^n + \sum_{\text{$k,n$:\,odd}}f(k,n)t^k q^n\notag \\
&=\sum_{\text{$k,n$:\,even}}f(k,n)(t^2)^{\frac{k}{2}} q^n + \sum_{\text{$k,n$:\,odd}}f(k,n)t^{-1}(t^2)^{\frac{k+1}{2}} q^n, \label{Ftq}
\end{align}
and we have
\begin{align*}
F(i,-iq)
&=\sum_{\text{$k,n$:\,even}}f(k,n)(-1)^{\frac{k}{2}}  (-1)^{\frac{n}{2}}q^n + \sum_{\text{$k,n$:\,odd}} f(k,n)(-1)^{\frac{k+1}{2}} (-1)^{\frac{n+1}{2}} q^n\\
&=\sum_{\text{$n$:\,even}}(-1)^{\frac{n}{2}} \big(o(n) -e(n)\big) q^n + \sum_{\text{$n$:\,odd}}(-1)^{\frac{n+1}{2}} \big(o(n) -e(n) \big)  q^n \\
&=\sum_{n \ge 0}(-1)^{ \lceil \frac{n}{2} \rceil } \big(o(n) -e(n) \big)  q^n,
\end{align*}
where the second equality follows from Lemma~\ref{lem1}.
\end{proof}

Throughout this paper, we  adopt the following standard $q$-series notation:
\begin{equation*}
(a;q)_{\infty}:=\lim_{n\to \infty} \prod_{j=0}^{n-1}(1-aq^j),
\end{equation*}
and the following compressed notation:
\begin{equation*}
(a_1,\ldots, a_m;q)_{\infty}:=(a_1;q)_{\infty} \cdots (a_m;q)_{\infty}.
\end{equation*}

In what follows, we will see some special cases of Theorem~\ref{thm2}. 

\begin{cor}
We have 
$$
\sum_{n\ge 0}(-1)^{ \lceil \frac{n}{2} \rceil } \big(o_{\pa}(n) -e_{\pa}(n) \big)q^n = \frac{1}{(q,-q^2;-q^2)_\infty}.
$$
\end{cor}
\begin{proof}
Since $F_{\pa}(t,q)=  \dfrac{1}{(tq,q^2;q^2)_\infty}$, we have $F_{\pa}(i,-iq)=  \dfrac{1}{(q,-q^2;-q^2)_\infty}$.
\end{proof}

\begin{cor}
Let $\dis:=\{\lam\in \pa | \mbox{ $\lam_i$'s are distinct}\}$. Then
$$
\sum_{n\ge 0}(-1)^{ \lceil \frac{n}{2} \rceil } \big(o_{\dis}(n) -e_{\dis}(n) \big)q^n = (-q,q^2;-q^2)_\infty.
$$
\end{cor}
\begin{proof}
Since $F_{\dis}(t,q)=  (-tq,-q^2;q^2)_\infty$, we have $F_{\dis}(i,-iq)=  (-q,q^2;-q^2)_\infty$.
\end{proof}

\begin{cor} \label{cor3}
For $d\ge 1$, let $\opa^{d}:=\{\lam\in \pa \,  | \mbox{ $\lam_i$'s are congruent to $\pm$1 (mod $4d$)}\}$. Then
$$
\sum_{n\ge 0}(-1)^{ \lceil \frac{n}{2} \rceil } \big(o_{\opa^{d}}(n) -e_{\opa^{d}}(n) \big)q^n = \frac{1}{(q,-q^{4d-1};q^{4d})_\infty}.
$$
\end{cor}
\begin{proof}
Since  $F_{\opa^{d}}(t,q)=  \dfrac{1}{(tq,tq^{4d-1};q^{4d})_\infty}$,  $F_{\opa^{d}}(i,-iq)=  \dfrac{1}{(q, -q^{4d-1}; q^{4d})_\infty}$
\end{proof}

Let $\opa$ be the set of partitions into odd parts.  Clearly, $\opa=\opa^{1}$, so $E_{\ind}(n)= o_{\opa}(n) -e_{\opa}(n)$.  Therefore, Theorem~\ref{thm1} is a special case of Corollary~\ref{cor3}.

\section{ $\ind_{1^n}(\lam)$} \label{sec4}

In \cite{collmayersmayers}, V. Coll, A. Mayers, and N. Mayers also considered seaweed algebras of type $\frac{\lambda}{1^{|\lambda|}}$. For $k\ge 0$, let 
$$
c_n(k):=\# \{ \lambda\vdash n \,|\,   \ind_{1^n} (\lambda) =n-k-1\}
$$   
and
$$
c(k)=\lim_{n\to \infty} c_{n}(k).
$$
\begin{thm}[Coll, Mayers, Mayers \cite{collmayersmayers}] \label{thm3}
We have
$$
\sum_{k\ge 0} c(k) x^{k}=\frac{1}{(x;x)_{\infty}^2}.
$$
\end{thm}

They prove this theorem by constructing a bijection between the set of partitions counted by $c(k)$ and the set of two colored partitions. This result can also be proved by generating function manipulations. 

First, note that in the meander associated with $\frac{\lambda}{1^n}$, there are no cycles.  Thus, by \eqref{index} and \eqref{p}, we get
$$
\ind_{1^{n}}(\lambda)= P-1=\frac{\op(\lambda)+n}{2}-1.
$$ 
So, by the definition of $c_n(k)$, we see that
\begin{equation}
c_n(k)=f_{\pa}(n-2k,n).\label{cnk}
\end{equation}

Let $f_i$ be the number of parts of size $i$ in $\lambda$. Then
\begin{equation*}
n=\sum_{i\ge 1}\lambda_i=\sum_{i\ge 1}  i  f_i= \sum_{i\ge 1} f_{2i-1} + 2\sum_{i\ge 1} i (f_{2i}+f_{2i+1}) . 
\end{equation*}
Since  $\lambda$ has $n-2k$ odd parts,  we get 
\begin{equation*}
k=\sum_{i\ge 1} i(f_{2i}+f_{2i+1}) \ge \sum_{i\ge 1} (f_{2i}+f_{2i+1}),
\end{equation*}
which shows that  $\lambda$ can have at most $k$ parts of size greater than $1$ and the largest part of $\lambda$ cannot exceed $2k+1$.  Thus, for a given $k$, $c_n(k)$ stabilizes  as $n\to \infty$, and its limit $c(k)$ counts the number of partitions with parts greater than $1$ such that  $\sum_{i\ge 1} i(f_{2i}+f_{2i+1})=k,$ i.e.,
\begin{equation*}
c(k)=\# \{\lambda \,|\, f_1=0, \;\,  \sum_{i\ge 1} i(f_{2i}+f_{2i+1})= k\}.
\end{equation*}

Define $\tc_n(k)$ by 
\begin{equation}
\tc_n(k) = \# \{\lambda \vdash n\,|\, f_1=0,\;\, \sum_{i\ge 1} i(f_{2i}+f_{2i+1})=k \}. \label{cntilde}
\end{equation}
It is clear that $\sum_{n \ge 0} \tc_n(k) = c(k)$.

Let us consider the generating functions for $c_n(k)$ and $\tc_n(k)$.  By \eqref{cnk}, we see that $c_0(0)=1$ and $c_n(k)=0$ if $n<2k$. 

\begin{thm} \label{thm:cnk}
We have
\begin{align*}
\sum_{ k,n\ge 0}  c_{n}(k) t^{k}q^{n}&=\frac{1}{(q,tq^2;tq^2)_{\infty}},\\
\sum_{ k,n\ge 0} \tc_n(k)  t^{k} q^{n} &=\frac{1}{(tq^2,tq^3;tq^2)_{\infty}}.
\end{align*}
\end{thm}
\proof
Note that 
\begin{align*}
\sum_{ k,n\ge 0} c_n(k) t^{n-2k}q^n &= \sum_{ k,n\ge 0} f_{\pa}(n-2k,n) t^{n-2k}q^n\\
& =F_{\pa}(t,q)  \\
&= \frac{1}{(tq,q^2;q^2)_{\infty}}.
\end{align*}
By substituting $t^{-1/2}$ and $t^{1/2}q$ for $t$ and $q$, respectively,  we get
$$
\sum_{ k,n\ge 0} c_n(k) t^{k} q^n 
= F_{\pa}(t^{-1/2},t^{1/2}q)
=\frac{1}{(q,tq^2;tq^2)_{\infty}}.
$$

We now prove the second identity. Let
\begin{equation*}
\tilde{F}_{\pa} (t,q) = \sum_{ k,n\ge 0} \tc_n(k) t^{n-2k} q^n. 
\end{equation*}
By \eqref{cntilde}, we see that the  partitions counted by $\tc_n(k)$ cannot have parts of size $1$. Thus
$$
\tilde{F}_{\pa}(t,q)=(1-tq)F_{\pa}(t,q).
$$
Hence,
\begin{equation*}
\sum_{ k,n\ge 0}\tc_n(k) t^{n-2k} q^n=(1-tq) F_{\pa} (t,q).
\end{equation*}
Upon substituting $t^{-1/2}$ and $t^{1/2}q$ for $t$ and $q$ above, we immediately get the second identity. 
\iffalse
\begin{align*}
\sum_{n\ge 0} \sum_{k\ge 1} \tc_n(k) t^{k-1} q^n 
& = (1-q) F_{\pa}(t^{-1/2},t^{1/2}q)\\
&
=\frac{1}{(tq^2, tq^3 ;tq^2)_{\infty}}, 
\end{align*}
which completes the proof. \fi
\endproof

When $q=1$, the second identity of Theorem~\ref{thm:cnk} yields Theorem~\ref{thm3}.

\section{Parity difference of $\ind_{1^n}(\lam)$}\label{sec5}

We first define the conjugated index of a partition $\lambda$. 
\begin{defn} 
For a partition $\lambda$ of $n$, we define the conjugated index of $\lambda$ by the index of the seaweed of type $\frac{\lambda}{1^n}$ and denote it by $\cind(\lambda)$, i.e.,
\begin{equation*}
\cind(\lambda):=\ind_{1^n}(\lambda).
\end{equation*}
For $n=0$, if $\lambda$ is the empty partition, we define $\ind_{1^n} (\lambda)$ by $-1$. 
\end{defn}

In this section, we will consider the generating function for partitions with the conjugated index weight. 

\begin{defn}
Given a subset $S$ of $\pa$, we define
\begin{align*}
    \be(n)=
    \be_{S}(n)&:= \# \{\lam\in S_n \, |\,   \mbox{$\cind(\lam)$ is even}\},\\
    \bo(n)=
    \bo_{S} (n)&:=\# \{\lam\in S_n \, |\,   \mbox{$\cind(\lam)$ is odd}\}.
\end{align*}
\end{defn}

We now consider the difference $\bo(n)-\be(n)$.

\begin{lem}  \label{lem2}
For any $n \ge 0$,
\begin{align*}
&(-1)^{\lfloor \frac{n}{2} \rfloor} \big(\bo(n)-\be(n) \big)  \\
&=\begin{cases}
\#\{\lam\in S_n : \op(\lam) \equiv 0 \Mod{4})\} -\#\{\lam\in S_n : \op(\lam) \equiv 2 \Mod{4}\} & \mbox{if $n$ is even,}\\ 
\#\{\lam\in S_n : \op(\lam) \equiv 3 \Mod{4}\} -\#\{\lam\in S_n : \op(\lam) \equiv 1 \Mod{4}\} & \mbox{if $n$ is odd.}
\end{cases}
\end{align*}
\end{lem}
\begin{proof}
Note that $\op(1^n)=n$. Thus it easily follows from \eqref{eq1}.
\end{proof}

Recall $f(k,n)$ and $F(t,q)$ in Definition~\ref{def:fF}.

\begin{thm} \label{thm5}
We have
$$
F(i,-iq)
= \sum_{n\ge 0} (-1)^n \big(\bo(n) -\be(n) \big)q^n ,
$$
where $i=\sqrt{-1}$.
\end{thm}
\begin{proof}
By \eqref{Ftq},  we see that
\begin{align*}
&F(i,-iq) \\
&=\sum_{\text{$k,n$:\,even}}f(k,n)(-1)^{\frac{k}{2}} (-1)^{\frac{n}{2}}q^n + \sum_{\text{$k,n$:\,odd}}f(k,n)(-1)^{\frac{k+1}{2}} (-1)^{\frac{n+1}{2}}q^n\\
&=\sum_{\text{$n$:\,even}}(-1)^{\frac{n}{2}} \big(\bo(n) -\be(n) \big) (-1)^{\frac{n}{2}} q^n + \sum_{\text{$n$:\,odd}}(-1)^{\frac{n-1}{2}} \big(\bo(n) -\be(n) \big)(-1)^{\frac{n+1}{2}} q^n \\
&=\sum_{n \ge 0}(-1)^n \big(\bo(n) -\be(n) \big) q^n,
\end{align*}
where the second equality follows from Lemma~\ref{lem2}. 
\end{proof}

By comparing Theorems~\ref{thm2} and \ref{thm5}, we get the following corollary. 

\begin{cor}
For any $n \ge 0$,
$$
o(n) -e(n) = (-1)^{\lfloor \frac{n}{2} \rfloor} \big(\bo(n) -\be(n) \big).
$$
In particular, 
$$
|o(n) -e(n)| = |\bo(n) -\be(n)|.
$$
\end{cor}

\begin{rmk} 
Besides the partitions $\mu=n$ and $\mu=1^n$, we  can consider other partitions of $n$.
\begin{enumerate}
   \item 
Let $\mu$ be a partition of $n$ satisfying that
$$
\begin{cases}
\op(\mu) \equiv 1 \Mod{4} & \mbox{if $n$ is odd,}\\ 
\op(\mu) \equiv 0 \Mod{4} & \mbox{if $n$ is even.}
\end{cases}
$$
For example, let $\mu=2^{\lfloor \frac{n}{2} \rfloor} 1^{n-2\lfloor \frac{n}{2} \rfloor}$. Then we have 
$$
\ind_{\mu}(\lam) \equiv \ind_{n}(\lam) \pmod{2}.
$$
Thus $\# \{\lam\in S_n \, |\,   \mbox{$\ind_{\mu}(\lam)$ is odd (resp.~even)}\}$ is the same as $o(n)$ (resp. $e(n)$). Moreover, the weight generating functions are the same.

\item 
Let $\mu$ be a partition of $n$ satisfying that
$$
\op(\mu) \equiv n \pmod{4}.
$$
For example,  let  $\mu=4^{\lfloor \frac{n}{4} \rfloor} 1^{n-4\lfloor \frac{n}{4} \rfloor}$. Then we have 
$$
\ind_{\mu}(\lam) \equiv \ind_{1^n}(\lam) \pmod{2}.
$$
Thus $\# \{\lam\in S_n \, |\,   \mbox{$\ind_{\mu}(\lam)$ is odd (resp.~even)}\}$ is the same as $\bo(n)$ (resp. $\be(n)$). Moreover, the weight generating functions are the same.

\end{enumerate}
\end{rmk}

\section{Concluding Remarks}

In this section, we provide some remarks along with further conjectures related to index generating functions.

As mentioned in Introduction, Conjecture~\ref{conj1} is equivalent to the following conjecture. 

\begin{cnj} For $n\ge 1$,
$$(-1)^{ \lceil \frac{n}{2} \rceil } \big(o_{\opa}(n) -e_{\opa}(n) \big)=|e_{\opa}(n) -o_{\opa}(n) |.$$
\end{cnj}

It seems true that the non-negativity holds even if we take $\opa^{d}$ for any $d\ge 1$. 

\begin{cnj} For $n\ge 0$, $d\ge 1$,
$$(-1)^{ \lceil \frac{n}{2} \rceil } \big(o_{\opa^d}(n) -e_{\opa^d} (n) \big)=|e_{\opa^d}(n) -o_{\opa^d}(n) |. $$
\end{cnj}
Equivalently, for any $d\ge 1$, 
\begin{equation}
\frac{1}{(q,-q^{4d-1};q^{4d})_{\infty}} \label{eq12}
\end{equation}
has non-negative coefficients.  In what follows, we list some conjectures that yields the non-negativity of the coefficients in \eqref{eq12}. 

\begin{cnj} For $m\ge 4$, every coefficient of
$$\dfrac{1}{(q,-q^{m-1};q^m)_\infty} $$ is nonnegative.
\end{cnj}

\begin{cnj}[increasing sequences by $m$] \label{conj5}
For $m \ge 4$, there exists a positive integer $N(m)$ such that  for any $n\ge N(m)$, 
$$
\left[q^n\right] \dfrac{1}{(q,-q^{m-1};q^m)_{\infty}} \ge \left[q^{n-m}\right] \dfrac{1}{(q,-q^{m-1};q^m)_{\infty}
},
$$
where  $[q^n]F(q)$ denotes the coefficient of $q^n$ in the series $F(q)$. 
\end{cnj}

\end{document}